\documentclass{amsart}

\usepackage[usenames]{color}
\usepackage{amssymb}
\usepackage{graphicx}
\usepackage{amscd}

\usepackage[colorlinks=true,
linkcolor=webgreen,
filecolor=webbrown,
citecolor=webgreen]{hyperref}

\definecolor{webgreen}{rgb}{0,.5,0}
\definecolor{webbrown}{rgb}{0,0,.5}

\usepackage{color}
\usepackage{fullpage}
\usepackage{float}

\usepackage{graphics,amsmath,amssymb}
\usepackage{amsthm}
\usepackage{amsfonts}
\usepackage{latexsym}
\usepackage{epsf}

\newcommand{\seqnum}[1]{\href{http://oeis.org/#1}{\underline{#1}}}

\theoremstyle{plain}
\newtheorem{theorem}{Theorem}
\newtheorem{corollary}{Corollary}
\newtheorem{lemma}{Lemma}
\newtheorem{proposition}{Proposition}

\theoremstyle{definition}
\newtheorem{definition}{Definition}
\newtheorem{example}{Example}
\newtheorem{conjecture}{Conjecture}
\renewcommand{\subjclassname}{AMS \textup{2020} Mathematics Subject Classification: 11B99, 11A07, 11A51}
\theoremstyle{remark}
\newtheorem{remark}{Remark}

\author{J.M. Grau}
\address{Departamento de Matemáticas, Universidad de Oviedo\\ Avda. Calvo Sotelo, s/n, 33007 Oviedo, Spain}
\email{grau@uniovi.es}
\author{A. M. Oller-Marc\'{e}n}
\address{Centro Universitario de la Defensa\\ Ctra. de Huesca, s/n, 50090 Zaragoza, Spain}
\email{oller@unizar.es}
\author{D. Sadornil}
\address{Departamento de Matemáticas, Estadística y Computación,
Universidad de Cantabria\\ F. Ciencias, Avda de los Castros s/n,
39005 Santander, Spain}
\email{daniel.sadornil@unican.es}

\title{On $\mu$-Sondow numbers}
\begin{document}
\maketitle

\begin{abstract}
Given an integer $\mu$, we study the numbers that satisfy the condition $\frac{\mu}{n} + \sum_ {p \mid n} \frac {1} {p} \in \mathbb{N}$. This condition, which is reminiscent of the one satisfied by Giuga numbers ($\mu=-1$), also includes the so-called \cite{sondow} weak primary pseudoperfect numbers ($\mu=1$). As a tribute to our late colleague Jonathan Sondow (1943 -- 2020), we have named these numbers $ \mu $-Sondow numbers. In this paper, we give several different characterizations of these numbers, all of them suggested by well-known characterizations of the Giuga numbers. We also relate these numbers to the well-known Erd\"os-Moser equation and we present some conjectures about them.
\end{abstract}

\subjclassname{}

\section{Introduction}

In  number theory, a \emph{primary pseudoperfect number} is an integer $n>1$ that satisfies the Egyptian fraction equation
\begin{equation}
\label{defppn}
\sum_{p \mid n}\frac{1}{p}+\frac{1}{n} =1.
\end{equation}
Primary pseudoperfect numbers were first introduced and investigated by Butske et al. in  2000 \cite{BUT}. Using computational search techniques, they proved the remarkable result that, for each positive integer $r$ up to 8, there exists exactly one primary pseudoperfect number with precisely $r$ distinct prime factors. To date, only 8 primary pseudoperfect numbers are known (sequence \seqnum{A054377} in OEIS). Namely,
$$2, 6, 42, 1806, 47058, 2214502422, 52495396602, 8490421583559688410706771261086.$$
It is interestiing to point out that, as it was observed by Sondow and MacMillan in 2017 \cite{sondow2}, if we reduce modulo 288 those primary pseudoperfect numbers with $2 \leq r \leq 8$, we get an arithmetic progression of difference $36$:
$$\{6, 42, 78, 114, 150, 186, 222\}.$$

Borwein et al. \cite{BOR} introduced the so-called \emph{Giuga numbers}. A Giuga number is a composite integer $n$ such that $p\mid (n/p-1)$ for every prime divisor $p$ of $n$. These numbers were named after the Italian mathematician Giuseppe Giuga, who first encountered them while searching for primality conditions \cite{GIU}. Up to date, only thirteen Giuga numbers are known (sequence \seqnum{A007850} in OEIS). The first twelve ones are
\begin{center}
\begin{tabular}{|c|c|}
\hline Giuga number & Prime factorization\\
\hline 30 & $2\cdot 3\cdot 5$\\
\hline 858 & $2\cdot 3\cdot 11\cdot 13$\\
\hline 1722 & $2\cdot 3\cdot 7\cdot 41$\\
\hline 66198 & $2\cdot 3\cdot 11\cdot 17\cdot 59$\\
\hline 2214408306 & $2\cdot 3\cdot 11\cdot 23\cdot 31\cdot 47057$\\
\hline 24423128562 & $2\cdot 3\cdot 7\cdot 43\cdot 3041\cdot 4447$\\
\hline 432749205173838 & $2 \cdot 3 \cdot 7 \cdot 59 \cdot 163 \cdot 1381 \cdot 775807$\\
\hline 14737133470010574 & $2 \cdot 3 \cdot 7 \cdot 71 \cdot 103 \cdot 67213 \cdot 713863$\\
\hline 550843391309130318 & $2 \cdot 3 \cdot 7 \cdot 71 \cdot 103 \cdot 61559 \cdot 29133437$\\
\hline 244197000982499715087866346 & $2 \cdot 3 \cdot 11 \cdot 23 \cdot 31 \cdot 47137 \cdot 28282147\cdot 3892535183$\\
\hline 554079914617070801288578559178 & $2 \cdot 3 \cdot 11 \cdot 23 \cdot 31 \cdot 47059 \cdot 2259696349 \cdot 110725121051$\\
\hline 1910667181420507984555759916338506 & $2 \cdot  3 \cdot 7 \cdot  43\cdot  1831 \cdot 138683 \cdot 2861051\cdot  1456230512169437$\\
\hline
\end{tabular}
\end{center}

One further Giuga number with 97 digits and 10 prime factors is known, namely:

{\scriptsize\begin{align*}  &4200017949707747062038711509670656632404195753751630609228764416142557211582098432545190323474818 = \\
 &= 2  \cdot 3  \cdot 11  \cdot 23  \cdot 31  \cdot 47059  \cdot 2217342227  \cdot 1729101023519 \cdot 8491659218261819498490029296021  \cdot 58254480569119734123541298976556403 \end{align*}}

The behavior modulo 288 of the primary pseudoperfect numbers observed by Sondow is somewhat reproduced in the context of Giuga numbers. If we reduce the first twelve Giuga numbers modulo $288$, we get the following sequence.
$$ 30, 282, 282, 246, 210, 210, 174, 174, 174, 138, 138, 138 $$
which, if we ignore the repetitions, is an arithmetic progression of difference $-36$. Moreover, the residue of a Giuga number modulo 288 seems to be closely related to the number of its prime factors. In fact, note that Giuga numbers with the same number of prime factors turn out to be equal when reduced modulo 288. Moreover, the 13th Giuga number is congruent to 66 modulo 288, and the fact that 102 is missing in the previous sequence might be related to the fact that no Giuga number with 9 prime factors is known.\footnote{See https://math.stackexchange.com/q/2432325}

Giuga numbers admit several characterizations, leading to apparently different definitions depending on the context. Some of them are summarized in the following proposition \cite{AGO,BOR,GIU,GOL,VARI}.
\begin{proposition}
\label{defsGiuga}
Let $n$ be a composite integer. Then, the following are equivalent:
\begin{itemize}
\item[i)] $p\mid (n/p-1)$ for every prime divisor $p$ of $n$.
\item[ii)] $\sum_{p\mid n}\frac{1}{p}-\prod_{p\mid n}\frac{1}{p}\in\mathbb{N}$.
\item[iii)] ${\sum_{j=1}^{n-1}  j^{\phi(n)} \equiv -1} \pmod n$, where $\phi$ is Euler's totient function.\footnote{The function $\phi$ can be replaced by the Carmichael lambda function \cite{VARI}.}
\item[iv)] $nB_{\phi(n)} \equiv -1  \pmod n$, where $B$ is a Bernoulli number.
\item[v)] $n'=an+1$ for some $a\in \mathbb{N}$, where $n'$ denotes the arithmetic derivative.\footnote{The arithmetic derivative \cite{BAR} of an integer $n=\prod_{i=1}^{k} p_i ^{n_i}$ is given by $ n'= n \sum_{i=1}^k\frac{n_i}{p_i}.$ This was first introduced by the Spanish mathematician Jos\'e Mingot Shelly in 1911 \cite{espanol}.}
\end{itemize}
\end{proposition}

It is clear from their definition that Giuga numbers are squarefree. Consequently, Proposition \ref{defsGiuga} ii) can be rewritten so that a Giuga number is a composite integer that satisfies the condition
\begin{equation}
\label{defgn}
\sum_{p \mid n}\frac{1}{p}-\frac{1}{n}\in\mathbb{N}.
\end{equation}

Now, the strong resemblance between \eqref{defppn} and \eqref{defgn} motivates the following definition, which yields a natural generalization of primary pseudoperfect numbers.

\begin{definition}
An integer $n$ is a \emph{weak primary pseudoperfect number} if it satisfies the condition
\begin{equation}
\label{defwppn}
\frac{1}{n}+\sum_{p \mid n}\frac{1}{p} \in \mathbb{N}.
\end{equation}
\end{definition}

These numbers were first introduced in \cite{sondow}, when studying the solutions to the congruence
$$1^{kn} + 2^{kn} + \dotsb + (kn)^{kn}\equiv n \pmod{kn}.$$
In that work it was proved that if $(k,n)$ is a solution of this congruence, then $k$ must be a weak primary pseudoperfect number. The converse, however, is not true since there is no solution with $k=52495396602$.

Obviously, primary pseudoperfect numbers are just a particular case of weak primary pseudoperfect numbers and, in fact, all known weak primary pseudoperfect numbers up to date happen to be also primary pseudoperfect numbers. It is easy to see that any weak primary pseudoperfect number which is not a primary pseudoperfect number must have more that 58 prime factors, and therefore must be greater than $10^{110}$. This might partly explain the difficulty in finding such an example \cite{sondow}.

The first goal of this paper is to provide an analogue to Proposition \ref{defsGiuga} for weak primary pseudoperfect numbers, leading to several different characterizations of this class of numbers. This is done is Section 3, after some technical results. Furthermore, we will see that Proposition \ref{defsGiuga} and our analogue (given in Theorem \ref{charwpsn}) share a natural generalization that leads to the introduction of a new class of numbers. We decided to call these numbers $\mu$-Sondow numbers in memory of our late colleague Jonathan Sondow (1943--2020). Recall that given an integer $n$, the $p$-adic order of $n$ is just $\nu_p(n)=\max\{s:p^s\mid n\}$. Then, we have the following.

\begin{definition}
Given $ \mu \in \mathbb{Z}$, a \textit{$\mu$-Sondow number} is an integer $n$ such that $p^{\nu_p(n)} \mid (n/p +\mu)$ for every prime $p \mid n$.
\end{definition}

With this definition, as we will see, Giuga numbers are just composite $(-1)$-Sondow numbers\footnote{The set of $(-1)$-Sondow numbers was registered at the OEIS by Jonathan Sondow a few months before his death \seqnum{A326715}. This motivated our definition.}, while weak primary pseudoperfect numbers are just $1$-Sondow numbers. Then, our second goal of this paper will be to characterize and study some general properties of $\mu$-Sondow numbers. This is done in Section 4. Finally, as an application, we relate $\mu$-Sondow numbers to the well-known Erd\"os-Moser equation \cite{moree2} in the last section of the paper.

\section{Some technical results}

In this section, we provide the main technical tools that will be required in order to prove our main result. We begin with a well-known lemma whose proof can be found in \cite[Theorem 1.]{KELL}. Recall that, if $r_1, r_2\in\mathbb{Q}$, the statement $r_1\equiv r_2\pmod{n}$ means that $n$ divides the numerator of $r_1-r_2$.

\begin{lemma}
\label{lema1}
Let $k$ and $n$ be positive integers with $k$ even, and $n>1$. Then,
$$
\sum_{i=1}^{n-1} i^k \equiv n\cdot B_k \equiv - \sum_{\substack{p \mid n\\ p-1 \mid k}}\frac{n}{p} \pmod{n},
$$
where $B_k$ is the $k$-th Bernoulli number.
\end{lemma}

The following lemma will also be useful in the sequel. 

\begin{lemma}
\label{foro1}
Let $\mu \in \mathbb{Z}$. Then $\displaystyle \sum_{i=1}^{n-1} i^{\phi(n)} \equiv \mu \pmod {n}$ if and only if $p^{\nu_p(n)} | (n/p + \mu)$ for every prime $p\mid n$.
\end{lemma}
\begin{proof}
Let us denote $S=\displaystyle \sum_{i=1}^{n-1}i^{\phi(n)}$.

Now, let us assume that $S\equiv \mu \pmod{n}$ and let $p\mid n$ be a prime. Then, $S\equiv \mu \pmod{p^{\nu_p(n)}}$ and, due to Lemma \ref{lema1}, we have that
$$\mu\equiv S\equiv - \sum_{\substack{p \mid n}}\frac{n}{p}\equiv -  \frac{n}{p} \pmod{p^{\nu_p(n)}},$$
so $p^{\nu_p(n)} | (n/p + \mu)$ as claimed.

Conversely, let us assume that $p^{\nu_p(n)} \mid (n/p + \mu)$ for every prime $p\mid n$. Consequently, $\mu\equiv -n/p\pmod{p^{\nu_p(n)}}$ for every prime $p\mid n$. Then, Lemma \ref{lema1} together with the Chinese remainder theorem imply that $S\equiv \mu\pmod{n}$, and the result follows.
\end{proof}

The last lemma that we are about to present involves the arithmetic derivative \cite{BAR}. The proof of this result can be found in \cite[Corollary 2]{VIC} and it is a rather direct consequence of the definition of the arithmetic derivative of a number $n$
$$n'=n\sum_{p\mid n}\frac{\nu_p(n)}{p}.$$

\begin{lemma}
\label{nada}
Let $\mu \in \{-1,1\}$. If $n'=an - \mu$ for some $a\in\mathbb{N}$, then $n$ is square-free.
\end{lemma}

\section{The weak primary pseudoperfect numbers}

After this short review of technical lemmas, we can provide an analogue to Proposition \ref{defsGiuga} in the context of weak primary pseudoperfect numbers.

\begin{theorem}
\label{charwpsn}
Let $n$ be a positive integer. Then, the following are equivalent.
\begin{itemize}
\item[i)]
$\frac{1}{n}+\sum_{p \mid n}\frac{1}{p} \in \mathbb{N}$.
\item[ii)]
$\sum_{p\mid n}\frac{n}{p}+1\equiv 0\pmod{n}$.
\item[iii)]
$n B_{\phi(n)}\equiv 1 \pmod{n}$.
\item[iv)]
$\sum_{i=1}^{n-1} i^{\phi(n)} \equiv 1 \pmod n$.
\item[v)]
$p \mid (n/p+1)$ for every prime divisor $p\mid n$, .
\item[vi)]
$n'=an-1$, for some $a \in \mathbb{N}$.
\end{itemize}
\end{theorem}
\begin{proof}
\ \
\begin{itemize}
\item[i) $\Rightarrow$ ii)] Trivial.
\item[ii) $\Rightarrow$ iii)] Just apply Lemma \ref{lema1}, with $k=\phi(n)$, taking into account that $p-1\mid\phi(n)$ for every $p\mid n$.
\item[iii) $\Rightarrow$ iv)] Lemma \ref{lema1} with $k=\phi(n)$, again.
\item[iv) $\Rightarrow$ v)] Just apply Lemma \ref{foro1} with $\mu=1$.
\item[v) $\Rightarrow$ vi)] Let us assume that $p \mid (n/p+1)$ for every prime $p\mid n$. First, this implies that $n$ is squarefree, so that $n'=\sum_{p\mid n}\frac{n}{p}$. Furthermore, it also implies that $n\mid \prod_{p\mid n}(n/p+1)$. Now, it is rather straightforward to see that, for a squarefree $n$, $\prod_{p\mid n}(n/p+1)=An+n'+1$ for some integer $A$ and the claim follows.
\item[vi) $\Rightarrow$ i)] Let us assume that $n'=an-1$, for some $a \in \mathbb{N}$. Lemma \ref{nada} (with $\mu=1$) implies that $n$ is squarefree, and $n'=n\sum_{p\mid n}\frac{1}{p}$. Thus, $an-1=n'=n\sum_{p\mid n}\frac{1}{p}$ implies that $\displaystyle{\sum_{i=1}^k\frac{1}{p_i} +\frac{ 1}{n}=a\in \mathbb{N}}$, as claimed.
\end{itemize}
\end{proof}

\begin{remark}\label{rem1}
It is interesting to compare Theorem \ref{charwpsn} with Proposition \ref{defsGiuga}. First of all, we note that Theorem \ref{charwpsn} does not require $n$ to be composite. However, it is easily seen that $2$ is the only prime number that is also a weak primary pseudoperfect number. On the other hand, it is noteworthy that all the conditions from Theorem \ref{charwpsn} are counterparts of conditions from Proposition \ref{defsGiuga} just by substituting a $(-1)$ for a $1$. We will get back to this idea later.
\end{remark}

We end the section showing that a well-known property of primary pseudoperfect numbers still remains true for weak primary pseudoperfect numbers. This proposition can be proved using any of the equivalent conditions from Theorem \ref{charwpsn}.

\begin{proposition}
Let $n$ be a weak primary pseudoperfect number such that $n+1$ is prime. Then, $n(n+1)$ is a weak primary pseudoperfect number.
\end{proposition}
\begin{proof}
Let us assume that $n$ is a weak primary pseudoperfect number. Then, due to Theorem \ref{charwpsn} i), we have that $\displaystyle \frac{1}{n} + \sum_{p|n}\frac{1}{p}=a$ for some $a\in\mathbb{N}$.

Now,
$$\frac{1}{n(n+1)} + \sum_{p|n(n+1) }\frac{1}{p}=\frac{1}{n(n+1)} + \sum_{p|n }\frac{1}{p}+\frac{1}{n+1}=\frac{1}{n} + \sum_{p|n}\frac{1}{p}=a$$
and the result follows.
\end{proof}

\section{The $\mu$-Sondow numbers}

As we pointed out in Remark \ref{rem1}, there is a very strong resemblance between Theorem \ref{charwpsn} and Proposition \ref{defsGiuga}. This suggests a possible generalization in the following direction, that we already presented in the introduction.

\setcounter{definition}{1}
\begin{definition}
Given $ \mu \in \mathbb{Z}$, a \textit{$\mu$-Sondow number} is an integer $n$ such that $p^{\nu_p(n)} \mid (n/p +\mu)$ for every prime $p \mid n$.
\end{definition}

As expected, $\mu$-Sondow numbers admit several equivalent characterizations that we provide in the following theorem.

\begin{theorem} \label{charmu}
Let $\mu \in \mathbb{Z}$, and let $n$ a positive integer. Then, the following are equivalent.
\begin{itemize}
\item[i)]
For every prime $p\mid n$, $p^{\nu_p(n)} \mid (n/p+\mu)$.
\item[ii)]
$\sum_{i=1}^{n-1} i^{\phi(n)} \equiv \mu \pmod n$.
\item[iii)]
$n B_{\phi(n)}\equiv \mu \pmod{n}$.
\item[iv)]
$\frac{\mu}{n}+\sum_{p \mid n}\frac{1}{p} \in \mathbb{Z}$.
\item[v)]
$\sum_{p\mid n}\frac{n}{p}+\mu\equiv 0\pmod{n}$.
\end{itemize}
\end{theorem}
\begin{proof} \ \
\begin{itemize}
\item[i) $\Longleftrightarrow$ ii)] Lemma \ref{foro1}.
\item[ii) $\Longleftrightarrow$ iii)] Lemma \ref{lema1} with $k=\phi(n)$.
\item[ii) $\Longleftrightarrow$ iv)] Lemma \ref{lema1} with $k=\phi(n)$, and taking into account that $p-1\mid\phi(n)$ for every $p\mid n$.
\item[iv) $\Longleftrightarrow$ v)] Trivial.
\end{itemize}
\end{proof}

In what follows, we will denote by $\mathfrak{S}_{\mu}$ the set of $\mu$-Sondow numbers. In addition, we will consider $\mathbb{G}$ the set of Giuga numbers, $\mathbb{W}$ the set of the weak primary pseudoperfect numbers, and $\mathbb{P}$ the set of prime numbers. The following result is an easy consequence of Theorem \ref{charmu} a provides the relationship between these sets.

\begin{corollary}
Let $n>1$ be a positive integer. Then,
\begin{itemize}
\item[i)]  $n \in \mathbb{G}$ if and only if $n$ is composite and $n\in \mathfrak{S}_{-1}$.
\item[ii)]  $n \in \mathbb{W}$  if and only if $n\in \mathfrak{S}_{1}$.
\end{itemize}
In other words,
$$\mathfrak{S} _{1} = \mathbb{W},\textrm{ and } \mathfrak{S} _{-1} =  \{1\} \cup \mathbb{P} \cup \mathbb{G}.$$
\end{corollary}

The following results provide a method to construct $\mu$-Sondow numbers starting from Giuga numbers or weak primary pseudoperfect numbers. Recall that the radical of a positive integer $n$, $rad(n)$, is just the product of the distinct primes dividing $n$.

\begin{proposition}
Let $\mu>1$ be a positive integer. Then, $\mu n \in  \mathfrak{S}_{\mu}$ if and only if $rad(\mu) | n$ and $ n \in  \mathfrak{S} _{1}$.
\end{proposition}
\begin{proof}
First, let us assume that $\mu n \in  \mathfrak{S}_{\mu}$. If $rad(\mu) \nmid n$, there exists a prime $p \mid \mu$ such that $p \nmid n$. Hence, $\nu_p(\mu n)=\nu_p(\mu)$ but $p^{\nu_p(\mu n)} \nmid (n \mu /p + \mu)$, a contradiction. On the other hand, if $n \notin \mathfrak{S} _{1}$ then there exists a prime $p$ such that $p\mid n$ such that $p \nmid (n/p+1)$. Since $\mu n\in \mathfrak{S}_{\mu}$, $p^{\nu_p(\mu n)}\mid (n \mu /p + \mu)=\mu(n/p+1)$. Thus, $p^{\nu_p(\mu n)}\mid \mu$ so $p^{\nu_p(\mu n)+1}\mid\mu n$, which is also a contradiction.

Conversely, let us assume that $rad(\mu) | n$ and $n \in \mathfrak{S} _{1}$. Note that $n \in \mathfrak{S} _{1}$ implies that $n$ is square-free, for if $p^2|n$ for some prime $p$, then $p$ divides $n/p$ and also $n/p+1$, which is impossible. Now, since $rad(\mu)\mid n$ and $n$ is square-free, for every prime $p\mid \mu n$ it follows that $p^{\nu_p(\mu n)-1}\mid \mu$ and also that $p^{\nu_p(\mu n)}\mid (n \mu/p+ \mu  )$. Thus, $n\mu \in \mathfrak{S}_{\mu}$ as claimed.
\end{proof}

\begin{proposition}
Let $\mu>1$ be a positive integer. Then, $\mu n \in  \mathfrak{S}_{-\mu}$ if and only if $rad(\mu) | n$ and $n \in \mathbb{G} \cup \mathbb{P}$.
\end{proposition}
\begin{proof}
Almost identical to the proof of the previous proposition.
\end{proof}

\begin{remark}
The two propositions above can be unified in the following way. Let $\mu$ be an integer with $|\mu|>1$. Then, $|\mu| n \in  \mathfrak{S}_{\mu}$ if and only if $rad(\mu) | n$ and $n \in \mathfrak{S}_{\mu/|\mu|}$.
\end{remark}

Now, the following result goes, in some sense, in the opposite direction.

\begin{proposition}
Let $n$ be a positive integer such that $n \in  \mathfrak{S}_{\mu}$, and let $\delta=\gcd(n,\mu)$. Then, $n/\delta$ is square-free and $n/\delta\in \mathfrak{S}_{\mu/\delta}$.
\end{proposition}
\begin{proof}
Put $\displaystyle n=\prod_{p|n} p^{\nu_p(n)}$ and let us assume that $n \in  \mathfrak{S}_{\mu}$.

Let $p\mid n$ be a prime such that $p\nmid \mu$. Since $n \in  \mathfrak{S}_{\mu}$, it follows that $p^{\nu_p(n)}\mid n/p+\mu$. Consequently, $p^{\nu_p(n)-1}\mid \mu$ and $\nu_p(n)=1$.

On the other hand, let $p\mid n$ be a prime such that $p\mid \mu$. Reasoning again in the same way, we get that $p^{\nu_p(n)-1}\mid \mu$ and, consequently, $\nu_p(n)\leq \nu_p(\mu)+1$.

Let us assume for a moment that $\nu_p(n)<\nu_p(\mu)+1$. Then, since $n \in  \mathfrak{S}_{\mu}$
$$p^{\nu_p(n)}\mid n/p+\mu=p^{\nu_p(n)-1}\Big( n'+ p^{\nu_p(\mu)+1-\nu_p(n)}\mu' \Big)$$
for some $n'$ and $\mu'$ coprime to $p$. Then, it follow that $p\mid n'+ p^{\nu_p(\mu)+1-\nu_p(n)}\mu'$ and, if $\nu_p(n)<\nu_p(\mu)+1$, we get that $p\mid n'$. This is a contradiction which means that it must be $\nu_p(n)=\nu_p(\mu)+1$ and the result follows.
\end{proof}

We close this section by providing a series of examples in which we apply the previous propositions.

\begin{example}[$\mu = 8$]
In this case $ 8\mathbb {N} \cap \mathfrak {S} _{8} = 8 (2 \mathbb{N} \cap \mathbb{W}) $. All known weak primary pseudoperfects numbers (except $1$) are even, and when multiplied by 8 they give rise to $8$-Sondow numbers. Moreover, these are exactly the only $8$-Sondow numbers which are multiples of 8.
$$\mathbb{W}=\{1, 2, 6, 42, 1806, 47058, 2214502422, 52495396602, 8490421583559688410706771261086,\dots\}$$
$$\mathfrak{S}_8=\{1, 3, 8 \cdot 2, 8 \cdot 6, 8 \cdot42, 8 \cdot 1806, 8 \cdot 47058, 8 \cdot 2214502422, 8 \cdot 52495396602, 8 \cdot 8490421583559688410706771261086,\dots\}$$
\end{example}

\begin{example}[$ \mu = -5 $]
The only $(-1)$-Sondow numbers up to $1.9 \times 10 ^{33}$ which are multiples of $5$ are $\{5,30\}$. Thus, we can say that $\{25,150\} $ are the only $(-5)$-Sondow numbers multiples of 5 up $ 1.9 \times 5 \times 10^{33}$.
\end{example}

\begin{example}[$\mu = 5$]
No known weak primary pseudoperfect number is a multiple of 5. Hence, we cannot construct any $5$-Sondow numbers using them as a starting point and we can say that there are no $5$-Sondow numbers multiples of 5 up to 5 times the highest known weak primary pseudoperfect number (approximately $4 \times 10^{31} $).
\end{example}

\section{$\mu$-Sondow Numbers, the Erdős-Moser Equation, and some open Problems}

The Erdős–Moser conjecture states that the Diophantine equation $ 1^k+2^k+...+ (m-1)^k =m^k$, has no solution for positive integers $k$ and $m$ with $k>1$. This is still an open problem, even if it is known that possible solutions must have rather big values for $m$. In fact, it is known that a solution of the Erdős-Moser equation must have \cite{moree1}
$$m > 2.7139 \times 10^{1,667,658,416}.$$
The following result regarding possible solutions of the Erdős-Moser equation can be found in \cite[Theorem 4]{moree2} and gives a slightly improved version of Moser's original proof.

\begin{theorem}
Suppose that $(m,k)$ is a solution of the Erdős-Moser equation with $k\geq 2$. Then,
\begin{itemize}
\item[i)] $m>1.485\times 10^{9321155}$.
\item[ii)] $k$ is even, $m\equiv 3\pmod{8}$, $m\equiv \pm1\pmod{3}$.
\item[iii)] $m-1$, $(m+1)/2$, and $2m+1$ are all square-free.
\item[iv)] If $p$ divides at least one of the above integers, then $(p-1)\mid k$.
\item[v)] The number $(m^2-1)(4m^2-1)/12$ is square-free  and has at least $4990906$ prime factors.
\end{itemize}
\end{theorem}

The bound in i) is not the best known but the best Moser's method yields.

A close look at the proof of this result reveals a relation between possible solutions of the Erdős-Moser equation and $\mu$-Sondow numbers for $\mu \in \{1,2,4\}$. In fact, the following holds.

\begin{corollary}
Suppose that $(m,k)$ is a solution of the Erdős-Moser equation with $k\geq 2$. Then,
\begin{itemize}
\item[i)] $m+1,2m-1\in \mathfrak{S}_{2}$.
\item[ii)] $m-1\in \mathfrak{S}_{1}$.
\item[iii)] $2m+1\in \mathfrak{S}_{4}$.
\end{itemize}
\end{corollary}
\begin{proof}
See the proof of Theorem 4 in \cite{moree2} and recall the definition of $\mu$-Sondow numbers.
\end{proof}

We close this section and the paper with some open problems related to the existence of  $\mu$-Sondow numbers. It is obvious that for every positive integer $n$ there exists $ 0<\mu<n$ such that $n \in \mathfrak{S} _ \mu$ and $ n \in \mathfrak{S}_{-(n-\mu)}$. In addition, it is also straightforward that $\mathfrak{S}_{\mu} \neq \emptyset$ for every integer $\mu$ because we always have that $ 1 \in \mathfrak {S}_{\mu}$. However, the existence of $\mu$-Sondow numbers different from $1$ is an open problem. In this regard, we present the following conjecture.

\begin{conjecture}\label{conj}
\begin{itemize}
\item[i)] For every integer $\mu\notin\{0,1,-1,2,-2,4,16\}$, the set $\mathfrak{S}_{\mu} \cap [2,|\mu|]$ is not empty.
\item[ii)] For every integer $\mu$, the set $(|\mu|,\infty) \cap \mathfrak{S}_{\mu}$ is not empty.
\end{itemize}
\end{conjecture}

We note that we have been able to check computationally that Conjecture \ref{conj} ii) is true for every integer $-145<\mu<673$. However, for $\mu=-145$ and for $\mu=673$ we got $(|\mu|,10^{10}] \cap \mathfrak{S}_{\mu} =\emptyset$.

\section*{Acknowledgment}
The authors wish to thank Pieter Moree for his useful comments and suggestions, that helped us to improve the paper. Daniel Sadornil is partially supported by the Spanish Government under Project PID2019-110633GB-I00 from  MCIN/AEI/10.13039/501100011033

\end{document}